\documentclass[11pt,a4paper]{article}

\usepackage{graphicx}        

\usepackage[a4paper, hmargin={2.7cm,2.7cm},vmargin={3.3cm,3.3cm}]{geometry}
\usepackage{amsmath, amssymb, dsfont} 
\usepackage{amsthm}

\usepackage[T1]{fontenc}
\usepackage[utf8]{inputenc}
\usepackage[english]{babel}
\usepackage{cite}
\usepackage{paralist}
\usepackage{url}
\usepackage{twoopt}

\usepackage[bf,compact,pagestyles,small]{titlesec}
\titlespacing*{\section}{0pt}{14pt}{4pt}
\titlespacing*{\subsection}{0pt}{8pt}{3pt}

\usepackage{etoolbox}
\makeatletter
\patchcmd{\ttlh@hang}{\parindent\z@}{\parindent\z@\leavevmode}{}{}
\patchcmd{\ttlh@hang}{\noindent}{}{}{}
\makeatother

\usepackage{mathdots}

\usepackage{fancyhdr,lastpage}
\def\maketimestamp{\count255=\time
\divide\count255 by 60\relax
\edef\thetime{\the\count255:}%
\multiply\count255 by-60\relax
\advance\count255 by\time
\edef\thetime{\thetime\ifnum\count255<10 0\fi\the\count255}
\edef\thedate{\number\day-\ifcase\month\or Jan\or Feb\or Mar\or
             Apr\or May\or Jun\or Jul\or Aug\or Sep\or Oct\or
             Nov\or Dec\fi-\number\year}
\def\timstamp{\hbox to\hsize{\tt\hfil\thedate\hfil\thetime\hfil}}}
\maketimestamp
\fancypagestyle{plain}{%
\fancyhf{} 
\fancyfoot[C]{\small \thepage{} of \pageref{LastPage}}}

\numberwithin{equation}{section}  
\allowdisplaybreaks[4]

\newtheorem{theorem}{Theorem}[section]

\newtheorem{lemma}[theorem]{Lemma}

\theoremstyle{definition}
\newtheorem{definition}[theorem]{Definition} 

\theoremstyle{remark}
\newtheorem{remark}[theorem]{Remark}

\DeclareMathOperator{\exponential}{e}

\newcommandtwoopt{\mixedS}[2][\cG][\cH]{S_{{#1},{#2}}} 

\newcommandtwoopt{\gaborG}[3][\alpha][\beta]{\mathcal{G}(#3,#1,#2)} 
\newcommandtwoopt{\MD}[3][a][b]{\mathcal{MD}(#3,#1,#2)} 

\newcommand{\myexp}[1]{\exponential^{#1}}

\newcommand*{\numbersys}[1]{\ensuremath{\mathbb{#1}}}
\newcommand*{\C}{\numbersys{C}}
\newcommand*{\R}{\numbersys{R}}

\newcommand*{\Q}{\numbersys{Q}}

\newcommand*{\Z}{\numbersys{Z}}

\newcommand*{\cH}{\mathcal{H}}

\newcommand*{\cK}{\mathcal{K}}

\newcommand*{\cG}{\mathcal{G}}

\newcommand*{\cMD}{\mathcal{MD}}

\newcommand{\itvoo}[2]{\ensuremath{\left({#1},{#2}\right)}} %
\newcommand{\itvco}[2]{\ensuremath{\left[{#1},{#2}\right)}} %
\newcommand{\abs}[1]{\ensuremath{\left\lvert#1\right\rvert}}

\newcommand{\norm}[2][]{\ensuremath{\left\lVert#2\right\rVert_{#1}}}

\newcommand{\innerprod}[3][]{\ensuremath{\left\langle #2,#3\right\rangle_{\! #1}}}

\newcommand{\set}[1]{\ensuremath{\left\lbrace{#1}\right\rbrace}}

\newcommand{\setprop}[2]{\ensuremath{\left\lbrace{#1} : {#2}\right\rbrace}}

\usepackage{xspace}

\newcommand*\oline[1]{%
  \vbox{%
    \hrule height 0.5pt
    \kern0.25ex
    \hbox{%
      \kern-0.1em
      \ifmmode#1\else\ensuremath{#1}\fi
      \kern-0.1em
    }
  }
}

\usepackage{xcolor}

\usepackage{hyperref}
\hypersetup{
    colorlinks,
    linkcolor={red!50!black},
    citecolor={blue!50!black},
    urlcolor={blue!80!black},
 pdfview={FitH},
 pdfstartview={FitH},
 pdfauthor = {Jakob Lemvig, Kamilla Haahr Nielsen}
 pdftitle = {Counterexamples to the B-spline conjecture for Gabor
   frames}, 
 pdfkeywords = {B-spline, frame, frame set, Gabor system},
 pdfcreator = {LaTeX with hyperref package},
 pdfproducer = PDFlatex}

\makeatletter
\def\blfootnote{\xdef\@thefnmark{}\@footnotetext} 
\def\subjclass{\xdef\@thefnmark{}\@footnotetext}
\long\def\symbolfootnote[#1]#2{\begingroup%
\def\thefootnote{\fnsymbol{footnote}}\footnote[#1]{#2}\endgroup} 
\if@titlepage
  \renewenvironment{abstract}{%
      \titlepage
      \null\vfil
      \@beginparpenalty\@lowpenalty
      \begin{center}%
        \bfseries \abstractname
        \@endparpenalty\@M
      \end{center}}%
     {\par\vfil\null\endtitlepage}
\else
  \renewenvironment{abstract}{%
      \if@twocolumn
        \section*{\abstractname}%
      \else
        \small
        \list{}{%
          \settowidth{\labelwidth}{\textbf{\abstractname:}}
          \setlength{\leftmargin}{50pt}
          \setlength{\rightmargin}{50pt}
          \setlength{\itemindent}{\labelwidth}
          \addtolength{\itemindent}{\labelsep}
        }
        \item[\textbf{\abstractname:}]

      \fi}
      {\if@twocolumn\else\endlist\fi}
\fi
\makeatother

\begin{document}

\title{A remark on dilation-and-modulation frames for $L^2(\R_+)$}

\date{\today}

 \author{Jakob Lemvig\footnote{Technical University of Denmark, Department of Applied Mathematics and Computer Science, Matematiktorvet 303B, 2800 Kgs.\ Lyngby, Denmark, E-mail: \protect\url{jakle@dtu.dk}}} 

 \blfootnote{2010 {\it Mathematics Subject Classification.} Primary
   42C15. Secondary: 42A60}
 \blfootnote{{\it Key words and phrases.} dilation, frame, Gabor
   system, modulation}

\maketitle

\thispagestyle{plain}
\begin{abstract} 
  We show that every rationally sampled dilation-and-modulation system
  is unitarily equivalent with a multi-window  Gabor system. 
 As
    a consequence, frame theoretical results from Gabor
  analysis can be directly transferred to dilation-and-modulation systems.  
\end{abstract}

\section{Introduction}
\label{sec:introduction}

In  a recent series of papers~\cite{MR4025543,MR3529179,MR3820190,1708.05941},
so-called modulation-and-dilation systems (defined below) have
been suggested as a way to analyze causal time signals in the space $L^2(\R_+)$ of square integrable functions
defined on the positive real half-line $\R_+$. We will show
by a non-standard ``change of variable''-trick that the theory of
dilation-and-modulation systems in $L^2(\R_+)$ is a 
variant of the well-studied time-frequency systems considered in Gabor
analysis in
$L^2(\R)$.   

Let $a,b > 1$. We say that a function $f$ on $\R$ or $\R_+$ is \emph{$b$-dilation
periodic} if $f(\cdot)=f(b\cdot)$ and define a sequence of $b$-dilation periodic functions
$\gamma_m:\R_+ \to \C$ by
\begin{equation}
  \label{eq:1}
  \gamma_m(x) = 
\myexp{2\pi i m x/(b-1)} \qquad x \in \itvco{1}{b} 
\end{equation}
for each $m \in\Z$. The functions $\gamma_m$ are quite different from the
usual complex exponentials $x\mapsto \myexp{2\pi i bm x}$ and their
dilates. E.g., while $\gamma_m$ \emph{is} continuous, it is only piecewise
$C^\infty$ on $\R_+$ with discontinuity points of the derivative $\gamma_m'$, $m \neq 0$, at the $b$-adic fractions
$b^k$, $k \in \Z$.   
 
The dilation-and-modulation system ($\cMD$-system) generated by
$h_\ell \in L^2(\R_+)$, $\ell=1,\dots,L$,
is a countable family of $L^2(\R_+)$ defined by:
\begin{equation}
  \label{eq:MD-def}
  \MD{\set{h_\ell}_{\ell=1}^L} = \set{a^{j/2} \gamma_m(\cdot)
    h_\ell(a^j \cdot)}_{j,m \in \Z,\ell \in \set{1,\dots,L}}.
\end{equation}
If there is only one generator, i.e., $L=1$, we write the $\cMD$-system as $\MD{h}$.
A $\cMD$-system is said to be rationally sampled whenever
$\log_b(a)\in \Q$.
The objective of this note is to show any $\cMD$-system with 
$\log_b(a)\in \Q$ is unitarily equivalent with a regular Gabor system of the form
$\setprop{\myexp{2 \pi i m x} g_r(\cdot-\alpha k)}{m,k \in \Z,r=1,\dots,R}$ in $L^2(\R)$
for some constant $\alpha >0$ and generators
$\{g_1,\dots,g_R\} \in L^2(\R)$ dependent on $a$, $b$, and $h\in
L^2(\R_+)$. As a direct consequence, most properties of
interest, e.g., any basis or frame theoretical property, of
$\MD{h}$ in $L^2(\R_+)$
can be analyzed by Gabor theory of
$\setprop{\myexp{2 \pi i m x} g_r(\cdot-ck)}{m,k \in \Z,r=1,\dots,R}$ in $L^2(\R)$.
Thus, while rationally oversampled systems $\MD{h}$ can be of
practical importance, they should be considered as a variant of Gabor systems. 

The analysis of dilation-and-modulation systems in the literature~\cite{MR4025543,MR3529179,MR3820190,1708.05941} has so far been
restricted to the case of rational sampling $\log_b(a)\in \Q$. Interestingly, 
the here presented link to Gabor analysis seemingly breaks down when $\log_b(a)\notin \Q$. 
\section{Preliminaries}
\label{sec:preliminaries-1}


For a bounded
function $\eta \in L^\infty(\R)$, the 
 multiplication operator $M_\eta$ on $L^2(\R)$ is defined as $M_\eta f = \eta f$ for
$f \in L^2(\R)$. For $\eta(x)=\myexp{2 \pi i b x}$ with $b\in \R$, the
modulation operator $M_\eta$ is by slight abuse of notation written
$M_b$. For $a>0$ and $c\in \R$, the dilation operator
$D_a: L^2(\R)\to L^2(\R)$ and translation operator
$T_c: L^2(\R)\to L^2(\R)$ are defined by:
\[ 
    D_af(x) = a^{1/2} f(a x) \quad \text{and} \quad T_c f(x)=f(x-c)
\] 
for $x \in \R$, respectively. The multiplication and dilation operator
on $L^2(\R_+)$ is defined similarly, while the translation operator
is undefined on $L^2(\R_+)$.  

In terms of these operators, the $\cMD$-system~\eqref{eq:MD-def} generated by $h \in
L^2(\R_+)$ can  be written as
\begin{equation}
  \label{eq:MD-def-operator}
  \MD{h} = \set{M_{\gamma_m} D_{a^j}h}_{j,m \in \Z}.
\end{equation}
For $\alpha,\beta>0$ the \emph{(multi-window) Gabor system} 
generated by the functions $g_1,\dots,g_L \in L^2(\R)$ is given by 
\begin{equation}
  \label{eq:Gabor-def-operator}
  \gaborG[\alpha][\beta]{\set{g_1,\dots,g_L}} := \set{M_{\beta m}
    T_{\alpha k} g_\ell }_{k,m \in \Z, \ell \in \set{1,\dots, L}}.
\end{equation}
The three cases $\alpha \beta < 1$, $\alpha \beta =1$ and $\alpha
\beta >1$ are usually called oversampling, critically sampling and
undersampling, respectively, as these cases determine if sampling
time-frequency shifts of $g \in L^2(\R)$ in phase space can lead to frame/complete
systems. If $\alpha\beta \in \Q$, the sampling is said to be rational.

The ``change of variables''-trick used to turn $\cMD$ systems into Gabor
systems relies on the following
function. Let $\varphi: \R \to \R_+$ be a piecewise linear function interpolating
the sampling points $(k,b^k)$ for $k \in \Z$. Explicitly, for each $k
\in \Z$, we have:
\[ 
   \varphi(x) = b^k\bigl((b-1)x+(1-(b-1)k)\bigr) \quad \text{for } x \in \itvco{k}{k+1}.
\]
The slope of $\varphi$ grows exponentially as $\varphi'(x)= b^k(b-1)$ for $x \in \itvoo{k}{k+1}$
for every $k \in \Z$. 
 We define the linear operator $D_\varphi: L^2(\R_+) \to L^2(\R)$  by
$D_\varphi h = \sqrt{\varphi'} \cdot (h \circ \varphi)$ for all $h
\in L^2(\R_+)$, that is,   
$(D_\varphi h)(x) =  \sqrt{\varphi'(x)}\cdot h(\varphi(x))$  for a.e.\ $x
\in \R$ .
By the change of variable formula,
the operator $D_\varphi$ is an isometry. As $\varphi$ is a bijection, the
operator $D_\varphi$ is a bijection and thus unitary. Note that if $\varphi$ is a linear,
i.e.,  $\varphi(x)=a\, x$ (defined either as a mapping $\R\to \R$ or
$\R_+\to \R_+$), then $D_\varphi$ is simply the dilation
operator $D_a$; the similarity of notation should not lead to
confusions.

\begin{lemma}
\label{lem:communt}
Let $a,b>1$ be given such that $\log_b(a) \in \Q$, and let
 $q/p=\log_b(a)$ with $p,q \in \Z_{>0}$ being relatively prime. 
Then, as operators from $L^2(\R_+)$ to $L^2(\R)$, we have the
commutator relations: 
\begin{align}
    \label{eq:commut_dila}
    D_\varphi D_{a^{sq}} &= T_{-sp} D_\varphi   
\\
    \label{eq:commut_modu}
    D_\varphi M_{\gamma_m} &= \myexp{2\pi i m/(b-1)} M_{m} D_\varphi  
\end{align}
for any $s,m \in \Z$.
\end{lemma}
\begin{proof}
Let $h \in L^2(\R_+)$. We first prove \eqref{eq:commut_dila}.
By definition, for a.e.\ $x \in \R$, 
 \[ 
  (D_\varphi D_{a^j} h)(x) = a^{j/2} \varphi'(x)^{1/2}\, h(a^j \varphi(x))
  \quad\text{and}\quad 
  (T_{-\ell} D_\varphi h)(x) = \varphi'(x+\ell)^{1/2}\, h(\varphi(x+\ell)) ,
 \]
for any $j,\ell \in \R$.
Hence, setting $j=sq$ and $\ell=sp$, it suffices to  show:
\begin{equation}
a^{sq} \varphi(x) = \varphi(x+sp)\label{eq:varphi-prop1}
\end{equation}
 for $x \in \R$ and $s \in \Z$. 

 Let $s\in \Z$ be given. For each $k\in \Z$ and any $x \in \itvco{k}{k+1}$, we have 
  \begin{align*}
    a^{sq} \varphi(x) &= a^{sq} b^k\bigl((b-1)x+(1-(b-1)k)\bigr)
\intertext{and}
    \varphi(x+sp) &= b^{k+sp}\bigl((b-1)(x+sp)+(1-(b-1)(k+sp))\bigr) \\
                 &= b^{sp} b^k \bigl((b-1)x+(1-(b-1)k)\bigr). 
  \end{align*}
The proof of 
\eqref{eq:varphi-prop1} is complete if $a^{sq} = b^{sp}$ for $s \in
\Z$, but this follows directly from $q \log_b(a) = p$, i.e., $a^{q} = b^{p}$.

To see \eqref{eq:commut_modu}, we compute, for $m\in \Z$ and $x \in
\itvco{k}{k+1}\, (k\in \Z)$, 
\begin{align*}
  \gamma_m(\varphi(x)) &= \gamma_m(\varphi(x)/b^k) \\ &= 
 \myexp{2\pi i m
  [(b-1)x+(1-(b-1)k)]/(b-1)} \\
 &=  \myexp{2\pi i m (x-k+1/(b-1))} \\
 &=  \myexp{2\pi i m/(b-1)}\myexp{2\pi i m x}, 
\end{align*}
where we used the $b$-dilation periodicity of $\gamma_m$ and that
$\myexp{-2\pi i mk}=1$ for $k,m \in \Z$.
This, in turn, proves \eqref{eq:commut_modu}.
\end{proof}

There is a slightly delicate dependence on the dilation parameter in
\eqref{eq:commut_dila}. Indeed, commutator relations of the form
$D_\varphi D_{a^{s}} = T_{m} D_\varphi$, $s,m \in \R$, are only true
for the values of $s$ and $m$ described in Lemma~\ref{lem:communt}. 

\section{The equivalence of $\cMD$-systems and Gabor systems}
\label{sec:equivalence}
A \emph{frame} in a Hilbert space $\cH$ is a countable family of vectors
$\set{f_k}_{k \in I} \subset \cH$ for which there exist constants
$A,B>0$, called \emph{frame bounds}, so that
\[ 
  A \norm{f}^2 \le \sum_{k \in I} \abs{\innerprod{f}{f_k}}^2 \le B
  \norm{f}^2 \qquad \text{for all } f \in \cH.
\]
The largest such constant $A$ and smallest such constant $B$ are called
\emph{optimal frame bounds}.

\begin{definition}
\label{def:unitary-equi}
  Two countable families $\set{f_k}_{k \in I}$ and
  $\set{g_k}_{k \in I}$ in Hilbert spaces $\cH$ and $\cK$,
  respectively, are said to be \emph{unitarily equivalent} if there
  exists a unitary operator $U: \cH \to \cK$ and unimodular constants
  $c_k$ such that $g_k = c_k\, Uf_k$ for all $k \in I$.
\end{definition}
Basis and frame theoretical properties are
preserved by unitarily
equivalence, e.g., if $\set{f_k}_{k \in I}$ and $\set{g_k}_{k \in I}$
are unitarily equivalent, then $\set{f_k}_{k \in I}$ is a frame (or
Riesz basis)
precisely when $\set{g_k}_{k \in I}$ is a frame (or a Riesz basis), and the optimal
frame bounds are the identical. Hence, from a frame theoretical
point-of-view two unitarily equivalent systems are identical objects. 

We write $q/p=\log_b(a) \in \Q$ with $p,q \in \Z_{>0}$ relative prime.
For any $h \in L^2(\R_+)$, it follows by Lemma~\ref{lem:communt}  
that 
\begin{align}
    \label{eq:commut-MD-Gabor}
    D_\varphi M_{\gamma_m} D_{a^j} h = D_\varphi M_{\gamma_m}
  D_{a^{sq}} D_{a^{r}}
  h = \myexp{2\pi i m/(b-1)} M_{m}
                                     T_{-sp} D_\varphi D_{a^r} h
\end{align}
for any $j,m,s \in \Z$, where $j=sq+r$ with remainder $r\in \{0,1,
\dots, q-1\}$. From \eqref{eq:commut-MD-Gabor} it is 
straightforward to prove the following equivalence.

\begin{theorem}
\label{thm:frame-equi}
Let $a,b>1$ be given such that $\log_b(a) \in \Q$, and let
 $q/p=\log_b(a)$ with $p,q \in \Z_{>0}$ being relatively prime. Suppose $h \in L^2(\R_+)$. Then the following systems are unitarily equivalent:
\begin{enumerate}[(i)]
\item The dilation-and-modulation system $\MD{h}$ in $L^2(\R_+)$ \label{item:1}
\item The (multiwindow) Gabor system $\gaborG{\{g_1,\dots, g_q\}}$ in $L^2(\R)$,
  where $\alpha=p, \beta=1$ and $g_r = D_\varphi D_{a^{r}} h$ for
  $r=0,1,\dots, q-1$, \label{item:2}
\end{enumerate}
\end{theorem}

\begin{proof}
  The unitary equivalence follows immediately from \eqref{eq:commut-MD-Gabor}
  by, in Definition~\ref{def:unitary-equi}, identifying the unitary operator $U: \cH \to \cK$ as $D_\varphi: L^2(\R_+) \to
  L^2(\R)$, the vectors as $f_{(j,m)}=M_{\gamma_m} D_{a^j} h$ and
  $g_{(j,m)} = M_{m} T_{-sp} (D_\varphi D_{a^r} h)$,
   and the unimodular constants as $c_{(j,m)}=\myexp{2\pi i m/(b-1)}$
   for $(j,m)\in I:=\Z \times \Z$, where $j=sq+r$ with $r=0,\dots, q-1$.
\end{proof}

\begin{remark}
 It is straightforward to extend Theorem~\ref{thm:frame-equi} to 
 $\cMD$-systems with multiple generators
 $\MD{\set{h_\ell}_{\ell=1}^L}$. In this case, 
  the corresponding Gabor system
  simply becomes $\gaborG[p][1]{\{g_{1,1},\dots, g_{L,q}\}}$ with $Lq$ generators 
  given by $g_{\ell,r} = D_\varphi D_{a^{r}} h_\ell$ for
  $r=0,1,\dots, q-1$ and $\ell=1,\dots,L$. 
\end{remark}

The Gabor system  $\gaborG[p][1]{\{ g_1,\dots, g_q \}}$ in Theorem~\ref{thm:frame-equi}\eqref{item:2} is
integer \emph{undersampled} by a factor $p$, but
has $q$ generators:
\begin{equation}
g_1 = D_\varphi h, \quad
   g_2 = D_\varphi D_{a} h,\quad \cdots \quad g_{q} = D_\varphi
   D_{a^{q-1}} h. \label{eq:Gabor-gen-dilations}
 \end{equation}
In Gabor analysis this is a rather unusual form of the Gabor system, but
one should recall that even the standard rationally \emph{oversampled} Gabor system
$\gaborG[\alpha][\beta]{g}$, $\alpha \beta =p/q \in \Q$, can be
written as integer undersampled by a factor $p$ with $q$
generators since 
$\gaborG[\alpha][\beta]{g}$ is unitarily equivalent (by $D_\beta$) to:
\begin{equation*}
  \gaborG[p/q][1]{D_{1/\beta}g} = \set{M_{m}
    T_{p k/q} D_{1/\beta}g }_{k,m \in \Z} = \set{M_{m}
    T_{p k} \tilde{g}_r }_{k,m \in \Z, r\in \{0,1,\dots,q-1\}},
\end{equation*}
where
\begin{equation}
  \label{eq:Gabor-rat-oversampled-gen}
  \tilde{g}_r=T_r D_{1/\beta}g \quad \text{for $r\in \{0,1,\dots,q-1\}$}.
\end{equation}
For more information on Gabor analysis and, in particular, the results
mentioned in the discussion below, we refer to the standard text \cite{GroechenigFoundations2001}.

The theory of $\cMD$-systems developed in
\cite{MR4025543,MR3529179,MR3820190,1708.05941} is very
reminiscent of corresponding theory in Gabor
analysis, and Theorem~\ref{thm:frame-equi} provides a clear link between the two.  
Let us here restrict ourselves to a few examples of how
Theorem~\ref{thm:frame-equi} can be used to recover results from the theory of dilation-and-modulation systems. 

From the Density Theorem
in Gabor analysis together with Theorem~\ref{thm:frame-equi}, it
follows that $\log_b(a) \le 1$ is necessary for the frame
property, in fact, even for completeness, of rationally sampled dilation-and-modulation systems
$\MD{h}$. This recovers the density result of $\cMD$-systems proved in \cite{MR4025543}. 

The case $\log_b(a) = 1$, i.e., $a=b$, corresponds to
critically sampling of the dilation-modulation systems studied
in~\cite{MR3529179,1708.05941,MR3820190}. Indeed, since $a=b$
corresponds to $p=q=1$ and $\alpha=\beta=1$, the unitarily equivalent Gabor system
in Theorem~\ref{thm:frame-equi}(ii) becomes
\[
\gaborG[1][1]{D_\varphi h} = \set{M_m T_{s} (D_\varphi h)}_{m,s\in \Z}, 
\]
Hence, from well-known results in Gabor analysis on critically sampled
Gabor system, 
it follows that the frame property of
$\MD[a][a]{h}$ automatically implies that the $\cMD$-system is, in fact,
a Riesz basis; see \cite{1708.05941,MR3820190} for the direct proof in case of $\cMD$-systems.  

Theorem~\ref{thm:frame-equi} can also be used to obtain new results on
$\cMD$-systems. E.g., a Balian-Low type-theorem for $\cMD$-systems is
straightforward to formulate. From Theorem~\ref{thm:frame-equi}, it
follows directly from the classical Balian-Low theorem in Gabor
analysis that if $\MD[a][a]{h}$ is a frame (hence a Riesz basis) for $L^2(\R_+)$, then
the time-frequency uncertainty of $g:=D_\varphi h \in L^2(\R)$ is infinite, i.e.,
\[
  \biggl( \int_{-\infty}^\infty \abs{x-u}^2 \abs{g(x)}^2 dx \biggl)
  \cdot \biggl( \int_{-\infty}^\infty \abs{\omega-\eta}^2 \abs{\hat{g}(\omega)}^2 d\omega
  \biggl) = \infty, \quad \text{for any } u,\eta \in \R.
\]
The infinite uncertainty product of $D_\varphi h \in L^2(\R)$
translates into restrictions on
$h$ in $L^2(\R_+)$.

However, in the case of rational oversampling, i.e., $\Q \ni \log_b(a)<1$, not all
results from Gabor analysis can be directly transferred to
dilation-and-modulation systems. This holds, in particular, for
concrete examples of frames. E.g., it is, in general, not possible to
construct an $\cMD$-frame from an arbitrary rationally oversampled Gabor frame
for $L^2(\R)$, the reason being that the (multiple) generators of the
Gabor systems in Theorem~\ref{thm:frame-equi}(ii) are required to be
dilated versions of each other, compare~\eqref{eq:Gabor-gen-dilations}
to \eqref{eq:Gabor-rat-oversampled-gen},
which is a highly non-standard restriction in Gabor analysis. Finally, 
we remind the reader that the presented link by
Theorem~\ref{thm:frame-equi} seemingly breaks down whenever $\log_b(a)
\notin \Q$.




\end{document}